\pdfoutput=1
\documentclass{article}
\usepackage{xcolor}
\definecolor{blue5}{RGB}{1,31,51}
\definecolor{green4}{RGB}{21,103,84}
\usepackage[
  colorlinks=true,          % Color links instead of border
  urlcolor=blue5,           % Color of url-links
  linkcolor=blue5,          % Color of internal links (to equations for example)
  citecolor=green4,         % Color of bibliographic links
  ]{hyperref}
\usepackage{latexsym,amsmath,amscd,amssymb,framed,graphics}
\usepackage{amsthm}
\usepackage{thmtools, thm-restate}
\RequirePackage{enumitem}					% Custom labels for enumerated lists
	\setlist[enumerate,1]{label=(\roman*), ref=(\roman*)}		% (i), (ii),... 	
	\setlist[enumerate,2]{label=\alph*), ref=(\theenumi.\alph*)} 	% a), b) , ... and referenced as (i.a), (i.b),...
	\setlist[enumerate,3]{label=\arabic*., ref=(\theenumii.\arabic*)} 	% 1. , 2. , ... and referenced as (i.a.1)
\usepackage{graphicx}
\usepackage{esint}
\usepackage{microtype}      % Small typographic improvements
\usepackage[capitalise]{cleveref}
\usepackage[all]{xy}
\SelectTips{cm}{}
\usepackage{tikz-cd} 
\usetikzlibrary{hobby}
\RequirePackage{mathtools}	% Extend amsmath by little extras

\DeclareFontFamily{OMS}{fcmsy}{\skewchar\font48 }
\DeclareFontShape{OMS}{fcmsy}{m}{n}{%
         <-5.5> [.942] cmsy5     <5.5-6.5> [.942] cmsy6
      <6.5-7.5> [.942] cmsy7     <7.5-8.5> [.942] cmsy8
      <8.5-9.5> [.942] cmsy9     <9.5->  [.942] cmsy10
      }{}
\DeclareFontShape{OMS}{fcmsy}{b}{n}{%
       <-6> [.942] cmbsy5
      <6-8> [.942] cmbsy7
      <8->  [.942] cmbsy10
      }{}
\DeclareMathAlphabet{\mathcal}{OMS}{fcmsy}{m}{n}

\renewcommand{\tilde}{\widetilde} 
\renewcommand{\hat}{\widehat} 
\newcommand{\subeq}{\subseteq} 
\newcommand{\ssssarr}{\hbox to 15pt{\rightarrowfill}}
\newcommand{\sssarr}{\hbox to 20pt{\rightarrowfill}}
\newcommand{\ssarr}{\hbox to 30pt{\rightarrowfill}}
\newcommand{\sarr}{\hbox to 40pt{\rightarrowfill}}
\newcommand{\arr}{\hbox to 60pt{\rightarrowfill}}
\newcommand{\larr}{\hbox to 60pt{\leftarrowfill}}
\newcommand{\Arr}{\hbox to 80pt{\rightarrowfill}}

\newcommand{\ssmapright}[1]{\smash{\mathop{\ssarr}\limits^{#1}}}
\newcommand{\smapright}[1]{\smash{\mathop{\sarr}\limits^{#1}}}

\newcommand\de{\delta}

\newcommand\et{\eta}
\renewcommand\th{\theta}
\newcommand\io{\iota}

\newcommand\si{\sigma}

\newcommand\vphi{\varphi}
\newcommand\ph{\varphi}
\newcommand\ps{\psi}
\newcommand\om{\omega}

\newcommand\Om{\Omega}

\usepackage{xspace}
\newcommand\on{\operatorname}

\newcommand\ie{i.e.\ }
\newcommand\cf{cf.\@\xspace}

\newcommand\oo{{\infty}}

\renewcommand\o{\circ}

\newcommand\x{\times}

\newcommand\Emb{\on{Emb}}

\newcommand\curv{\on{curv}}
\newcommand\Flux{\on{Flux}}

\newcommand\ex{\on{ex}}

\newcommand\Aut{\on{Aut}}
\newcommand\SU{\on{SU}}

\newcommand\cP{\mathcal{P}}
\newcommand\cQ{\mathcal{Q}}

\renewcommand\P{\mathcal{P}}

\newcommand\Diff{\on{Diff}}

\newcommand\id{\on{id}}

\newcommand\ev{\on{ev}}

\newcommand\Hom{\on{Hom}}

\newcommand\ham{\on{ham}}

\newcommand\g{\mathfrak g}

\newcommand\tangent{\mathrm{T}}

\newcommand\pa{\partial}

\newcommand\Gr{\on{Gr}}

\newcommand\dd{{\tt d}}
\newcommand\dif{{\mathop{}\!\mathrm{d}}}

\newcommand\ZZ{\mathbb Z}

\newcommand\TT{\mathbb T}
\newcommand\RR{\mathbb{R}}
\newcommand\R{\mathbb{R}}
\newcommand\T{\mathcal{T}}
\newcommand\Z{\mathbb{Z}}

\newcommand\bC{\mathbb{C}}

\newcommand\X{\mathfrak X}

\NewDocumentCommand{\thmSep}{}{8.0pt plus 2.0pt minus 4.0pt}
\declaretheoremstyle[
		spacebelow=\thmSep,
		spaceabove=\thmSep,
		headfont=\bfseries,
		notefont=\normalfont,
		bodyfont=\normalfont\itshape,
		postheadspace=1em,
		qed=$\diamondsuit$,
		headpunct={}
]{myTheorem}

\declaretheoremstyle[
		spacebelow=\thmSep,
		spaceabove=\thmSep,
		headfont=\itshape,
		notefont=\normalfont,
		bodyfont=\normalfont,
		postheadspace=1em,
		qed=$\diamondsuit$,
		headpunct={}
]{myRemark}

\declaretheoremstyle[
		spacebelow=\thmSep,
		spaceabove=\thmSep,
		headfont=\bfseries,
		notefont=\normalfont,
		bodyfont=\normalfont,
		postheadspace=1em,
		qed=$\diamondsuit$,
		headpunct={}
]{mydef}

\declaretheorem[numberwithin=section, style=myTheorem]{theorem}

\declaretheorem[sibling=theorem, style=mydef]{definition}
\declaretheorem[sibling=theorem, style=myTheorem, name=Lemma]{lemm}
\declaretheorem[sibling=theorem, style=myTheorem, name=Proposition]{prop}
\declaretheorem[sibling=theorem, style=myTheorem, name=Corollary]{coro}
\declaretheorem[sibling=theorem, name=Problem, style=myTheorem]{pro}
	
\declaretheorem[sibling=theorem, style=myRemark, name=Remark]{rema}
\declaretheorem[sibling=theorem, style=myTheorem, name=Example]{exam}

\begin{document}

\title{Induced differential characters \\on nonlinear Gra\ss{}mannians}
\author{Tobias Diez, Bas Janssens, Karl-Hermann Neeb and Cornelia Vizman}

%% Keywords
%\keywords{Volume-preserving diffeomorphism, Nonlinear Gra\ss{}mannian, Tautological bundle,
% Holonomy-preserving diffeomorphism, 
%Differential character}

%% Mathematical classification (2010)
%\subjclass{53D20; 37K65; 58D10}

\maketitle
\date 

\begin{abstract}
Using a nonlinear version of the tautological bundle over Gra\ss{}mannians, 
we construct a transgression map for differential characters from \( M \) to the nonlinear Gra\ss{}mannian $\Gr^S(M)$ of submanifolds of \( M \) of a fixed type~$S$.
In particular, we obtain prequantum circle bundles of the nonlinear Gra\ss{}mannian endowed with the Marsden--Weinstein symplectic form.
The associated Kostant--Souriau prequantum extension yields central Lie group extensions of a
group of volume-preserving diffeomorphisms integrating Lichnerowicz cocycles. 
\end{abstract}

%% French abstract
%\begin{altabstract}
%En utilisant une version non-linéaire du fibr\'e tautologique sur les Gra\ss{}manniennes, nous construisons une application de transgression pour les caractères diff\'erentiels de \( M \) \`a la Gra\ss{}mannienne non-lin\'eaire $\Gr^S(M)$ des sous-vari\'et\'es de \( M \) d'un type fix\'e~$S$.
%En particulier, nous obtenons des fibr\'es en cercles préquantiques au dessus de la Gra\ss{}mannienne non-lin\'eaire dot\'e de la forme symplectique de Marsden--Weinstein.
%L'extension pr\'equantique de Kostant--Souriau associ\'ee donne des extensions  centrales de groupes de Lie d'un groupe de diff\'eomorphismes  pr\'eservant le volume et int\'egrant les cocycles de Lichnerowicz.
%\end{altabstract}    

\tableofcontents 

\pagebreak

%%%%%%%%%%%%%%%%%%%%
%%%%%%%%%%%%%%%%%%%%%
\section{Introduction}
The orbit method provides a powerful framework to construct irreducible unitary representations of an arbitrary Lie group.
In its simplest form, the method proceeds in two stages: first, construct an equivariant prequantum line bundle over certain coadjoint orbits and, second, pass to the space of sections that are covariantly constant relative to a chosen polarization.
Although originally developed in a finite-dimensional setting, the orbit method also has been successfully applied to infinite-dimensional Lie groups~\cite{PressleySegal1986,KirillovYuriev1988,Neeb2004a}. 
In this paper, we are concerned with the construction of prequantum bundles of a certain class of coadjoint orbits of the infinite-dimensional Lie group of volume-preserving diffeomorphisms.

Let \( (M, \mu) \) be a compact manifold of dimension \( n \geq 2 \) endowed with a volume form \( \mu \).
In~\cite{HV04,Ismagilov}, certain coadjoint orbits of the group of volume-preserving diffeomorphisms \( \Diff(M, \mu) \) were described in terms of the nonlinear Gra\ss{}mannian $\Gr^S(M)$ of all oriented submanifolds of $M$ of type $S$, 
where $S$ is a compact manifold of dimension $n-2$.
The tautological bundle over this nonlinear Gra\ss{}mannian
is a nonlinear version of the tautological vector bundle over the ordinary linear Gra\ss{}mannian. It is defined by
\[
\T := \bigl\{(N,x)\in\Gr^S(M)\x M\ :\  x\in N\bigr\},
\]
with bundle projection $q_1:\T\to\Gr^S(M)$, $ q_1(N,x)=N$.
We use the tautological bundle to define a transgression of a differential character \( h \in \widehat H^{n-1}(M,\TT) \)  to a differential character $\tilde h=(q_1)_!(q_2^*h) \in \widehat H^{1}(\Gr^S(M),\TT)$ on the nonlinear Gra\ss{}mannian, where $(q_1)_!$ denotes integration along the fibers of the tautological bundle and $q_2:\T\to M$ is defined by \( q_2(N, x) = x \).

Differential characters were introduced by Cheeger and Simons~\cite{CS85}, see also~\cite{BB} for a systematic exposition. 
Differential characters of degree one classify principal circle bundles with connections through their holonomy maps.
The map that associates to a differential character its curvature form is a surjective group homomorphism $\curv:\hat H^k(M,\TT)\to\Om_\ZZ^{k+1}(M)$ onto the group of differential forms with integral periods.
In this way, starting with a volume form $\mu\in\Om^n(M)$ that has integral periods,
we get via transgression a differential character of degree one on the nonlinear Gra\ss{}mannian \( \Gr^S(M) \).
This yields an isomorphism class of principal circle bundles $\cP \rightarrow \Gr^{S}(M)$ equipped with a connection 1-form $\Theta_{\cP} \in \Omega^1(\cP)$ whose curvature is the Marsden--Weinstein symplectic form $\tilde\mu$ induced by $\mu$~\cite{Marsden-Weinstein}.
That is, \( (\cP, \Theta_{\cP}) \) is a prequantization of \( (\Gr^S(M), \tilde\mu) \).
\begin{restatable}{extratheorem}{prequantum}
\label{prequantum}
	Let $M$ be a compact manifold of dimension \( n \) endowed with a volume form $\mu$ having integral periods, and let \( S \) be a closed, oriented manifold of dimension \( n-2 \).
	For every choice of a differential character \( h \in \widehat H^{n-1}(M,\TT) \) with curvature \( \mu \), the transgression $\tilde h \in \widehat H^{1}(\Gr^S(M),\TT)$ of \( h \) yields an isomorphism class of prequantum bundles of the nonlinear Gra\ss{}mannian \( \Gr^S(M) \) endowed with the Marsden--Weinstein symplectic form $\tilde\mu$.	 
\end{restatable}

Note that a prequantum bundle over the nonlinear Gra\ss{}mannian with curvature $\tilde\mu$ has been previously constructed in~\cite{HV04} and~\cite{Brylinski}.
These constructions yield a description of the prequantum bundle with connection in terms of local data.
An advantage of our approach is the additional control over the holonomy; for a fixed volume form $\mu$, the set of prequantum line bundles obtained with our construction is naturally a torsor over $H^{n-1}(M,\TT)$.

Associated to the volume form $\mu$, there is a flux homomorphism \break 
${\Flux_\mu:\Diff(M,\mu)_0\to J^{n-1}(M)}$ taking values in the Jacobian torus. 
The kernel $\Diff_{\ex}(M,\mu)$ of $\Flux_\mu$ acts on the nonlinear Gra\ss{}mannian while preserving its connected components.
We restrict the above constructed prequantum bundle $\cP \rightarrow \Gr^{S}(M)$ to the connected component \( \Gr_{N}^{S}(M) \) of $N\in\Gr^S(M)$.
Following ideas of Ismagilov~\cite{Ismagilov}, the pull-back of the prequantum extension
by the action of  $\Diff_{\ex}(M,\mu)$ yields a central Lie group extension of \( \Diff_{\ex}(M,\mu) \) that integrates the Lichnerowicz cocycle $\ps_N (X,Y) = \int_N i_X i_Y\mu$ on the Lie algebra $\X_{\ex}(M,\mu)$ of exact divergence free vector fields.
\begin{restatable}{extratheorem}{extensionsIntegratingLichnerowicz}
\label{extensionsIntegratingLichnerowicz}
Let $M$ be a compact manifold of dimension \( n \) endowed with a volume form $\mu$ having integral periods.
For every closed, oriented manifold \( S \) of dimension \( n-2 \) and for every differential character $h\in \widehat{H}^{n-1}(M,\TT)$ with curvature $\mu$, the 1-dimensional central extension \( \widehat\Diff_{\ex}(M,\mu) \) of $\Diff_{\ex}(M,\mu)$ obtained by pull-back of the prequantum extension~\eqref{eq:diag} is a Fr\'echet--Lie group that integrates the Lichnerowicz cocycle $\ps_N$.
\end{restatable}
\noindent Since $\Diff_{\ex}(M,\mu)$ acts transitively on connected components of $\Gr^S(M)$ according to~\cite[Prop.~2]{HV04}, this shows that \( \Gr_{N}^{S}(M) \) is a coadjoint orbit of \( \widehat\Diff_{\ex}(M,\mu) \).
A similar result for the identity component of \( \Diff_{\ex}(M,\mu) \) is obtained in~\cite[Thm.~2]{HV04}.

In~\cite{DJNV} we have used transgression of differential characters from $S$ and $M$ to get differential characters of degree one on the mapping space $C^\oo(S,M)$.
The associated central extension of $\Diff_{\ex}(M,\mu)$ integrates the Lichnerowicz cocycle as well.
In \cref{sec:comparison} we show that the two central extensions of $\Diff_{\ex}(M,\mu)$  constructed using transgression on the connected component of $f$ of the embedding space $\Emb(S,M)$ on one hand
and of the connected component of $f(S)$ of the nonlinear Gra\ss{}mannian $\Gr^{S}(M)$ on the other hand
are isomorphic as central extensions of Lie groups.

\paragraph*{Notation:}
We write $\TT = \{z \in \bC \ :\  |z| = 1\}$ for the circle group which we identify with $\R/\Z$.
Accordingly, we write $\exp_\TT(t) = e^{2\pi it}$ for its exponential function.

\paragraph*{Acknowledgment}
We would like to thank I. M\u{a}rcu\c{t} for pointing out the use of tautological bundles
over Gra\ss{}mannians, which is one of the main pillars of this paper.
C.~V. was supported by a grant of  the Romanian Ministry of Education and Research, CNCS-UEFISCDI,
project number PN-III-P4-ID-PCE-2020-2888, within PNCDI III.
B.~J. and T.~D. were supported by the NWO grant 639.032.734
``Cohomology and representation theory of infinite dimensional Lie groups''. 
K.-H.~N. acknowledges support by DFG-grant NE 413/10-1.

\section{Differential characters}
\label{s1}

In order to keep the paper self-contained,
we give a brief introduction to differential characters, following~\cite{CS85} and~\cite{BB}. 
The material in this section is essentially an abridged version of~\cite[\S2 and \S3]{DJNV}. 

\subsection{Basics on differential characters}

In this section $M$ denotes a locally convex smooth manifold 
for which the de Rham isomorphism holds\begin{footnote}%
{See~\cite[Thm.~34.7]{KrMi97} 
for a de Rham Theorem in this context and sufficient criteria for it to hold.}
\end{footnote}.
Let $C_k(M)$ be the group of smooth singular $k$-chains, and let $Z_k(M)$ and $B_k(M)$ denote the subgroups of $k$-cycles and $k$-boundaries,
so that $H_k(M) := Z_k(M)/B_k(M)$ is the $k$-th smooth singular homology group.

A \emph{differential character (Cheeger--Simons character) of degree $k$} 
is a group homomorphism 
$h:Z_k(M)\to\TT$ for which there exists a 
differential form $\om\in\Om^{k+1}(M)$ such that 
$$h(\pa z)= \exp_\TT\left(\int_z\om\right)$$
for all \( z\in C_{k+1}(M) \).
Then $\omega$ is uniquely determined by $h$ and is called the {\it curvature of $h$}, 
denoted by $\curv(h)$.
We write 
\[ \widehat H^{k}(M,\TT) \subeq \Hom(Z_{k}(M),\TT) \]
for the group of differential characters\begin{footnote}%
{In~\cite{BB} this group is denoted $\hat H^{k+1}(M,\Z)$. In this sense our notations 
are compatible, although the degree is shifted by $1$. Our convention 
follows the original one introduced by Cheeger and Simons in~\cite{CS85}.}  
\end{footnote}%
of degree $k$.
The curvature $\omega= \curv(h)$ satisfies
$1=h(\pa z)=\exp_\TT(\int_z\om)$ for all $z\in Z_{k+1}(M)$, so it
belongs to the abelian group of forms with integral periods
\[ \Omega_\ZZ^{k+1}(M) :=\Big\{ \omega \in \Omega^{k+1}(M) \,:\, \int_{Z_{k+1}(M)} 
\omega \subeq \Z\Big\}.
\]
On the other hand, the identification $H^{k}(M,\TT)\cong \Hom(H_{k}(M),\TT)$ yields a natural inclusion \( j: H^{k}(M,\TT) \to \widehat H^{k}(M,\TT) \), whose image is the subgroup of differential characters with zero curvature.
We get an exact sequence 
\begin{equation}
  \label{eq:1}
0\to H^{k}(M,\TT)\ssmapright{j}  \widehat H^{k}(M,\TT)\smapright{\curv}\Om_\ZZ^{k+1}(M)\to 0.
\end{equation}

\begin{rema}\label{one}
The holonomy map $h_{(P,\theta)}$ 
of a principal circle bundle $P\to M$ with connection form $\th\in\Om^1(P)$
assigns to each
piecewise smooth $1$-cycle $c\in Z_1(M)$ an element \( h_{(P, \theta)}(c) \) in $\TT$.
If the principal connection has curvature $\om\in\Om^2(M)$, then 
$h_{(P,\theta)}(\partial z) = \exp_\TT(\int_{z}\om)$ for all \( z \in C_2(M) \), so that  
$h_{(P,\theta)}\in \widehat H^1(M,\TT)$ is a differential character with 
$\curv(h_{(P,\theta)}) = \omega$.
The assignment $(P,\th)\mapsto h_{(P,\th)}$ defines an isomorphism between the group 
of isomorphism classes of pairs $(P,\theta)$ and the 
group $\widehat H^1(M,\TT)$ of differential characters of degree 1. 
To see this, note that principal circle bundles with connection are 
classified by Deligne cohomology~\cite[Thm.~2.2.12]{Brylinski}, which is an alternative 
model for differential cohomology, \cf~\cite[Sec.~5.2]{BB}.
For a direct proof, see also~\cite[Appendix~B]{DJNV}.
\end{rema}

\subsection{Stabilizer groups and Lie algebras.}
For any manifold \( M \), the action of the diffeomorphism group  \( \Diff(M) \) from the right on \( \widehat H^k(M, \TT) \) by pull-back~\cite[Rk.~15]{BB},
\begin{equation}\label{act}
	(\varphi^* h) (c) := h(\varphi \circ c) \quad \text{for } 
c\in Z_k(M),
\end{equation}
extends to the exact sequence \eqref{eq:1} of abelian groups:
\begin{equation}\label{lambda}
\xymatrix{
0\ar[r]&H^k(M,\TT)\ar[d]^{\vphi^*}\ar[r] ^j
&\widehat H^{k}(M,\TT) \ar[d]^{\vphi^*}\ar[r] ^{\curv}
& \Omega_{\Z}^{k+1}(M)\ar[d]^{\vphi^*} \ar[r] & 0\\
0\ar[r]& H^k(M,\TT)\ar[r] ^j&\widehat H^{k}(M,\TT) \ar[r]^{\curv} 
& \Omega_{\Z}^{k+1}(M) \ar[r] & 0.
}
\end{equation}
We denote the stabilizer group of the curvature form $\omega \in \Omega^{k+1}_{\Z}(M)$ by 
\begin{equation*}
\Diff(M,\om) := \{ \ph \in \Diff(M)\,:\, \ph^* \om = \om \}.
\end{equation*}
The stabilizer group of a differential character 
$h \in \widehat H^{k}(M,\TT)$,
\begin{equation*}
\Diff(M,h) := \{ \ph \in \Diff(M)\,:\, \ph^* h = h \} \,,
\end{equation*}
is a subgroup of $\Diff(M,\om)$ for $\om=\curv(h)$, by~\eqref{lambda}.
If $H_k(M) = \{0\}$, then $H^k(M, \TT)$ is trivial, and thus 
$\Diff(M,h) = \Diff(M,\om)$. 

\begin{rema}\label{holo}
Let $h_{(P,\th)}\in \hat H^1(M,\TT)$ be the differential character defined by the holonomy
 of the principal $\TT$-bundle $P\to M$ with connection $\th$ as in Remark~\ref{one}. Then 
$\ph \in \Diff(M,h_{(P,\th)})$ if and only if, for every smooth 
loop $c$ in $M$, the holonomy of $c$ coincides with the 
holonomy of~$\ph\o c$. Since this is equivalent 
to the existence of a lift to a connection-preserving automorphisms $\tilde\ph \in \Aut(P,\theta)$ 
by~\cite[Thm.~2.7]{NV}, one can view $\Diff(M,h_{(P,\th)})$ as the group of 
\emph{liftable} diffeomorphisms, \cf~\cite{Kostant,Souriau}. 
\end{rema}

Although $\Diff(M,\omega)$ need not be a locally convex Lie group, we can still define 
its Lie algebra as follows.

\begin{definition}\label{Def:FakeLieAlgebra}
We call a curve $(\ph_t)_{t \in [0,1]}$ in $\Diff(M)$ \emph{smooth} if the map 
$(t,x) \mapsto (\ph_t(x),\ph^{-1}_{t}(x))$ is smooth.
For a subgroup $G \subseteq \Diff(M)$, we 
denote by $G_0$ the group of diffeomorphisms 
that are connected to the identity by a piecewise smooth path 
in \( G \). 
We denote by $\de^l \varphi$ the left logarithmic derivative 
\begin{equation}
  \label{eq:logder}
 \delta^l\ph_t(x) := {\frac{\dif}{\dif\tau}}\Big|_{t} \ph_t^{-1} \bigl(\ph_\tau(x)\bigr),
\end{equation}
yielding a curve of vector fields on $M$.
Then a Lie subalgebra $\g \subseteq \X(M)$ is the Lie algebra of 
$G \subseteq \Diff(M)$ if for 
every smooth curve $(\ph_t)_{t\in [0,1]}$ in $\Diff(M)$ 
with $\ph_0 =\id_M$, the curve $(\ph_t)_{t\in [0,1]}$ is contained in $G$
if and only if its logarithmic derivative 
$(\delta^l\ph_t)_{t\in [0,1]}$ is a curve in~$\g$.
\end{definition}

In this sense, the Lie algebra of 
$\Diff(M,\om)$ is the stabilizer Lie algebra
\begin{equation*}
\X(M,\om) := \bigl \{X \in \X(M) \ :\  L_{X}\om = 0 \bigr \}.
\end{equation*}

\subsection{Flux homomorphism.} The isotropy group  $\Diff(M,h)$ is the kernel of the flux cocycle 
\begin{equation}
  \label{eq:fluxh}
\Flux_h:\Diff(M,\om)\to H^k(M,\TT),\quad\Flux_h(\ph)=\ph^*h-h.
\end{equation}
The restriction of $\Flux_h$ to the identity component $\Diff(M,\om)_0$ takes values in the Jacobian torus 
$J^k(M)\cong\Hom(H_k(M),\RR)/\Hom(H_k(M),\ZZ)$:
\begin{equation}
\xymatrix{
\Diff(M,\om)_0\ar[d]^{\io}\ar[r]^{\Flux_\om}
& J^k(M)\ar[d]^{\exp_{\TT}} \\
\Diff(M,\om)\ar[r]^{\Flux_h} & 
H^k(M,\TT).\\
}
\end{equation}
We denote this restriction by $\Flux_\om$, since it depends only on $\om=\curv(h)$.
Indeed, we can express $\Flux_{\om}$ as 
\begin{equation}\label{eq:propereendje} \Flux_\om(\ph)=\left[\int_0^1i_{\de^l\ph_t}\om \, \dif t\right],
\end{equation}
where $(\ph_t)_{t \in [0,1]}$ is any smooth curve in $\Diff(M,\om)$ 
with $\ph_0 =\id_M$ and $\ph_1=\ph$ \cite{Calabi, Banyaga}.
To see that the expression in~\eqref{eq:propereendje} is 
indeed the restriction of 
$\Flux_h$, note that for all $c\in Z_k(M)$,
\begin{align*}
\exp_\TT\Flux_{\om}(\ph)(c)
&
=\exp_{\TT}\left(\int_c\int_0^1i_{\de^l\ph_t}\om \, \dif t\right)
=\exp_\TT\left(\int_\si\om\right)=h(\pa\si)\\
&=h(\ph\o c)-h(c)=(\ph^*h-h)(c)=\Flux_h(\ph)(c),
\end{align*}
where $\si$ is the $(k+1)$-chain swept out by the $k$-cycle $c$ under the path of diffeomorphisms $\{\ph_t\}$.
The kernel of $\Flux_\om$ is the group 
\begin{equation}\label{opt}
	\Diff_{\ex}(M,\om) :=  \Diff(M, h) \cap \Diff(M,\omega)_{0},
\end{equation}
which is independent of the choice of $h$ with $\curv(h)=\om$.
The groups $\Diff(M,h)$ and 
$\Diff_{\ex}(M,\om)$ have the same Lie algebra 
\[
\X_{\ex}(M,\om) := \bigl \{X \in \X(M,\omega)\ :\  i_{X} \om \text{ is exact} \bigr \}.
\] 

\begin{exam}\label{exem} If  $M$ is compact and $\om\in\Om^{k+1}_\ZZ(M)$, 
then the following special cases are of particular importance:
\begin{enumerate}
\item For $k=1$ and $\om$ a symplectic form, $\X_{\ex}(M,\om)$ is the Lie algebra 
$\X_{\ham}(M,\om)$ of {\it Hamiltonian vector fields} and $\Diff_{\ex}(M,\om)_0$  is the group 
$\Diff_{\ham}(M,\om)$ of {\it Hamiltonian diffeomorphisms}.

\item For $k=n-1$ and $\om= \mu$ a volume form,
we get the Lie algebra $\X_{\ex}(M,\mu)$ of {\it exact divergence free vector fields}
and the group $\Diff_{\ex}(M,\mu)_0$ of {\it exact volume-preserving diffeomorphisms}.
\end{enumerate}
The corresponding groups $\Diff_{\ham}(M,\om)$ and $\Diff_{\ex}(M,\mu)_0$ are Fr\'echet--Lie groups~\cite[Thm. 43.7, 43.12]{KrMi97}.
In both cases mentioned above, the same 
holds for the possibly non-connected groups $\Diff(M,h)$ and $\Diff_{\ex}(M,\om)$ with $\curv(h)=\om$ for $h\in\hat H^k(M,\TT)$, \cf~\cite[Prop.~3.8]{DJNV}.
\end{exam} 

%%%%%%%
%%%%%%

\section{Tautological bundle over nonlinear Gra\ss{}mannians}\label{s3}

\subsection{Transgression of differential forms}\label{s2.1}

Let $M$ be a finite dimensional manifold, and let $S$ be a compact oriented $k$--dimensional manifold. 
The nonlinear Gra\ss{}mannian 
$\Gr^S(M)$ of all compact, 
oriented, $k$--dimensional submanifolds of $M$ of type $S$ 
is a Fr\'echet manifold, \cf~\cite[Thm.~44.1]{KrMi97}. 
The tangent space of \( \Gr^S(M) \)
at a submanifold $N$ can be identified with the space 
of smooth sections of the normal bundle $TN^\perp=(TM|_N)/TN$.
The natural surjection 
\begin{equation}\label{bundle}
\pi:\Emb(S,M)\to\Gr^S(M),\quad\pi(f)=f(S),
\end{equation}
where the orientation on the submanifold $f(S)$ is 
chosen such that the diffeomorphism $f:S\to f(S)$ is orientation-preserving, 
defines a principal bundle \( \Emb(S,M) \to \Gr^S(M) \) with structure group 
$\Diff_+(S)$, 
the group of orientation-preserving diffeomorphisms of $S$, \cf~\cite[Thm.~44.1]{KrMi97}.

The \emph{transgression}, or \emph{tilda map},~\cite{HV04}
associates to any $n$--form $\om$ on $M$ an $(n-k)$--form
$\tilde\om$ on $\Gr^S(M)$ by
\begin{equation}\label{tide}
\tilde\om_N(\tilde Y_1,\dotsc,\tilde Y_{n-k})
:=\int_N\iota_N^* (i_{Y_{n-k}}\cdots i_{Y_1}\om),
\end{equation}
where $\io_N:N\hookrightarrow M$ is the inclusion. Here $\tilde Y_j$ are tangent vectors at $N\in\Gr^S(M)$, \ie sections
of $TN^\perp$, represented by sections $Y_j$ of $TM|_N$.
Moreover, $\iota_N^*(i_{Y_{n-k}}\cdots i_{Y_1}\om)\in\Omega^k(N)$ is defined by
\begin{equation}\begin{split}
	\bigl(\iota_N^*(i_{Y_{n-k}}\cdots \, &i_{Y_1}\om)\bigr)_x (X_1, \dotsc, X_k)
		\\
		&= \omega_{\iota_N(x)} \bigl(Y_1(x), \dotsc, Y_{n-k}(x), \tangent_x \iota_N (X_1), \dotsc, \tangent_x \iota_N (X_k)\bigr)
\end{split}\end{equation}
for \( x \in N \) and \( X_i \in T_x N \),
so it does not depend on the representatives $Y_j$ of $\tilde Y_j$.
Finally, integration in~\eqref{tide} is well-defined since $N\in\Gr^S(M)$ comes with an orientation.

The natural action of the group $\Diff(M)$ on 
the nonlinear Gra\ss{}mannian $\Gr^S(M)$ is given by $\ph\cdot N=\ph(N)$. 
With the notations $\tilde\ph$ for the diffeomorphism of \( \Gr^S(M) \) induced by the action of $\ph\in\Diff(M)$ on $\Gr^S(M)$,
and  $\tilde X$ for the infinitesimal action of $X\in\X(M)$,
the following functorial identities hold: 
\begin{equation}\label{fun}
\tilde \ph^*\tilde{\om}=\widetilde{\ph^*\om}, \qquad 
L_{\tilde X}\tilde\om=\widetilde{L_X\om}, \qquad 
i_{\tilde X}\tilde\om=\widetilde{i_X\om}, \qquad 
\dd\tilde\om=\widetilde{\dd\om}.
\end{equation}

Similarly, $S$ being oriented, the {\it hat map}~\cite{Vizman2} associates to any form $\om\in\Om^n(M)$ the form 
$\widehat\om\in\Om^{n-k}(\Emb(S,M))$ defined by
\begin{equation}\label{exte}
\widehat{\om}(Z_1,\dots, Z_{n-k}):=\int_S f^*(i_{Z_{n-k}}\dots i_{Z_1}\om),
\end{equation}
with \( Z_j\in T_{f}\Emb(S,M) = \Gamma(f^*TM) \).
It is easy to check that the hat map on $\Emb(S,M)$ and the tilda map on $\Gr^S(M)$ are related by 
\begin{equation}\label{ht}
\widehat\om=\pi^*\tilde\omega
\end{equation}
for every \( \omega \in \Omega^n(M) \).

\subsection{Tautological bundle}

A nonlinear version of the tautological bundle over the usual Gra\ss{}mannian is
the associated bundle  \( \T = \Emb(S, M) \times_{\Diff_+(S)} S \) over
the nonlinear Gra\ss{}mannian $\Gr^S(M)$,
a smooth bundle with typical fiber $S$.
The tautological bundle can also be expressed as
\[
\T=\bigl\{(N,x)\in\Gr^S(M)\x M\,:\,x\in N\bigr\},
\]
with bundle projection $q_1:\T\to\Gr^S(M)$, $ q_1(N,x)=N$.
From this point of view, the quotient map $\Pi:\Emb(S,M)\x S\to\T$,
an $S$-bundle morphism over $\pi:\Emb(S,M)\to\Gr^S(M)$, becomes $\Pi(f,s)=\bigl(f(S),f(s)\bigr)$,
and the projection $q_2:\T\to M$ defined by 
$q_2(N,x)=x$ satisfies $q_2\o\Pi=\ev$.
Thus, the following diagrams commute:
\begin{equation}
	\label{eq:naturalTDiag}
\xymatrix{
\Emb(S,M)\x S\ar[d]^{p_1}\ar[r]^{\qquad\quad\Pi}
& \T\ar[d]^{q_1} \\
\Emb(S,M)\ar[r]^{\pi} & 
\Gr^S(M),\\
}
\qquad
\xymatrix{
\Emb(S,M)\x S\ar[dr]_{\ev}\ar[r]^{\qquad\quad\Pi}
& \T\ar[d]^{q_2} \\
& 
M.\\
}
\end{equation} 

The transgression~\eqref{tide} of a differential form $\om\in\Om^n(M)$  to the nonlinear Gra\ss{}mannian $\Gr^S(M)$ can be expressed with the help of the tautological bundle over the nonlinear Gra\ss{}mannian as
\begin{equation}\label{trans}
\tilde\om=(q_1)_!(q_2^*\om)\in \Om^{n-k}(\Gr^S(M)),
\end{equation}
where $(q_1)_!$ denotes integration along the fibers of the tautological bundle $q_1:\T\to\Gr^S(M)$.
Indeed, since $\pi$ is a submersion, this relation follows from
\[
\pi^*\bigl((q_1)_!(q_2^*\om)\bigr)=(p_1)_!(\Pi^*q_2^*\om)
\stackrel{\eqref{eq:naturalTDiag}}{=}(p_1)_!(\ev^*\om)=\hat\om\stackrel{\eqref{ht}}{=}\pi^*\tilde\om,
\]
using that fiber integration commutes with pull-back~\cite{GHV}.

In a similar spirit, tautological bundles over manifolds of nonlinear flags in $M$, \ie nested sets of submanifolds of $M$,
have been used in~\cite{HV20} to handle the transgression of differential forms on $M$ to differential forms on the manifold of nonlinear flags.

\subsection{Transgression of differential characters}
\label{sec:transgression_of_differential_characters}

The pull-back and the fiber integration make sense also for differential characters~\cite[Ch.~7]{BB}.
This allows us to  define the {\it transgression of a differential character} 
$h\in\hat H^{n-1}(M,\TT)$ to the nonlinear Gra\ss{}mannian \( \Gr^S(M) \) with the help of the tautological bundle $\T$, in the same way as in formula~\eqref{trans}:
\begin{equation}\label{transh}
\tilde h=(q_1)_!(q_2^*h)\in\hat H^{n-k-1}(\Gr^S(M),\TT).
\end{equation}
The transgression map  for differential characters
\begin{equation}\label{tm}
\hat H^{n-1}(M,\TT)\longrightarrow\hat H^{n-k-1}(\Gr^S(M),\TT)
\end{equation}
 has functorial properties that we describe below. 

\begin{prop}\label{ttau}
The transgression map  for differential characters in~\eqref{tm}
makes the following diagram commutative:
\begin{equation*}
\xymatrix{
H^{n-1}(M,\TT)\ar[d]^{\,\tilde{ }}\ar[r]^j &\widehat H^{n-1}(M,\TT) \ar[d]^{\,\tilde{ }}\ar[r] ^{\curv}
& \Omega_{\Z}^{n}(M)\ar[d]^{\,\tilde{ }}\\
H^{n-k-1}(\Gr^S(M),\TT)\ar[r]^j &\widehat H^{n-k-1}(\Gr^S(M),\TT) \ar[r]^{\ \ \ \curv} & 
\Omega_{\Z}^{n-k}(\Gr^S(M)),\\
}
\end{equation*}
where the transgression \( \tilde{a} \in H^{n-k-1}(\Gr^S(M),\TT) \) of \( a \in H^{n-1}(M,\TT) \) is defined by \( \tilde{a} = (q_1)_!(q_2^* a) \).
In particular, $\curv(\tilde h)=\widetilde{\curv(h)}$.
\end{prop}

\begin{proof}
Using the compatibility of both  the pull-back and  the fiber integration 
of differential characters with the curvature 
explained in~\cite[Rk.~15, Def.~38]{BB},
we compute
\[
\curv(\tilde h)=\curv\bigl((q_1)_!(q_2^*h)\bigr)=(q_1)_!\curv(q_2^*h)=(q_1)_!q_2^*\curv(h)=\widetilde{\curv(h)}.
\]
This shows the commutativity of the right side of the diagram.

To prove the commutativity of the left side, we use~\cite[Prop.~48]{BB}:
\[
\widetilde{j(a)} = (q_1)_!\bigl(q_2^*j(a)\bigr)=(q_1)_!\bigl(j(q_2^*a)\bigr)=j\bigl((q_1)_!(q_2^*a)\bigr)=j(\tilde a),
\]
for all $a\in H^{n-1}(M,\TT)$.
\end{proof}

\begin{coro}\label{mag}
The transgression map $\Om^n(M)\to \Om^{n-k}(\Gr^S(M))$, $\om\mapsto \tilde\om$,
preserves the integrality of differential forms.
\end{coro}
\begin{proof}
Every integral form $\om\in\Om^n_\ZZ(M)$ is the curvature of a character $h\in\hat H^{n-1}(M,\TT)$.
By Proposition~\ref{ttau}, its transgression $\tilde\om\in\Om^{n-k}(\Gr^S(M))$ 
is the curvature of the transgressed character $\tilde h\in\hat H^{n-k-1}(\Gr^S(M),\TT)$, 
hence an integral form.
\end{proof}

Consider the natural action of \( \Diff(M) \) on \( \T \) defined by assigning to every diffeomorphism 
\( \ph \in \Diff(M) \) the diffeomorphism $\ph_\T:\T\to\T$ given by $\ph_\T(N,x)=\bigl(\ph(N),\ph(x)\bigr)$.
The following diagram commutes:
\begin{equation}\label{x}
\xymatrix{
\Gr^S(M)\ar[d]^{\tilde\ph}&\T\ar[l]_{\qquad q_1}\ar[d]^{\ph_\T}\ar[r] ^{q_2}
& M\ar[d]^{\ph} \\
\Gr^S(M) &\T\ar[l]_{\qquad q_1}\ar[r]^{q_2} & 
M.\\
}
\end{equation}

\begin{prop}
\label{prop:transgressionDiffeoEquiv}
The transgression map  for differential characters defined in~\eqref{tm}
is compatible with the action of $\Diff(M)$, that is, \( \widetilde{\ph^*h}=\tilde\ph^*\tilde h \) holds for every \( h\in\widehat H^{n-1}(M,\TT) \) and \( \ph\in\Diff(M) \).
\end{prop}
\begin{proof}
The claim follows from the direct calculation
\[
\widetilde{\ph^*h}=(q_1)_!(q_2^*\ph^*h)=(q_1)_!(\ph_\T^*q_2^*h)=\tilde\ph^*\bigl((q_1)_!(q_2^*h)\bigr)=\tilde\ph^*\tilde h,
\]
by~\cite[Def.~38]{BB} and~\eqref{x}.
\end{proof}

The case of a volume form $\om\in\Om^n(M)$ and \( k = n - 2 \), 
with $\tilde\om\in\Om^2(\Gr^S(M))$, has been considered in~\cite[Theorem 1]{HV04}, where
a principal circle bundle $(\P,\th)$ over $\Gr^S(M)$ with curvature $\tilde\om$ has been constructed 
through its \v Cech 1-cocycle. 
In our setting, we get such a prequantum bundle over $\Gr^S(M)$ using the transgression $\tilde h\in \hat H^1(\Gr^S(M),\TT)$
of a differential character $h \in \hat H^{n-1}(M,\TT)$ with curvature $\om$
\begin{footnote}
{We do not know whether the holonomy $h_{(\P,\th)}\in\hat H^1(\Gr^S(M),\TT)$ 
of the principal bundle constructed in~\cite{HV04} 
coincides with the transgression of a differential character
$h \in \hat H^{n-1}(M,\TT)$. The differential character 
$h_{(\P,\th)}$ may  differ 
from $\tilde h\in\hat H^1(\Gr^S(M),\TT)$ by an element in 
$H^1(\Gr^S(M),\TT)$.
}\end{footnote}.
This is described in the next theorem, a direct consequence of Proposition~\ref{ttau} and Remark~\ref{one}.
\prequantum

A \emph{hat product} of differential characters has been introduced in~\cite[Section~4]{DJNV}
yielding the transgression of a pair of differential characters from $S$ and from $M$
to a differential character on $C^\oo(S,M)$.
We specialize it here to a \emph{hat map} that assigns to every \( h\in\hat H^{n-1}(M,\TT) \) 
the differential character
\begin{equation}
\hat h:=(p_1)_!(\ev^* h) \in \hat H^{n-k-1}(\Emb(S,M),\TT),
\end{equation}
where \( p_1: \Emb(S, M) \times S \to \Emb(S,M) \) is the natural projection.
\begin{prop}
\label{prop:piRelationOfTransgressions}
Given a differential character $h$ on $M$, 
the differential character $\hat h$ on $\Emb(S,M)$ is $\Diff_+(S)$ invariant 
and coincides with the pull-back 
$\pi^*\tilde h$ 
of the differential character $\tilde h$ on $\Gr^S(M)$ 
under the map  \( \pi: \Emb(S, M) \to \Gr^S(M) \) from~\eqref{bundle}. 
\end{prop}
\begin{proof}
	By~\eqref{eq:naturalTDiag}, the transgression diagrams for $\Emb(S,M)$ and for $\Gr^{S}(M)$ are connected by the projection $\pi \colon \Emb(S,M) \rightarrow \Gr^{S}(M)$, yielding the commutative diagram
\begin{equation*}
\begin{tikzcd}[row sep=scriptsize, column sep=scriptsize]
\Emb(S,M)\times S \arrow[dd, "p_1", near start]\arrow[dr, "\ev"]\arrow[rr, two heads, "\Pi"]& 
& \mathcal{T}\arrow[dr, "q_2"]\arrow[dd, "q_1",near start] & \\ 
& M \arrow[rr, equal, crossing over] & & M\\
\Emb(S,M)\arrow[rr, two heads, "\pi"]& & \Gr^{S}(M). & 
\end{tikzcd} 
\end{equation*}
	Accordingly, we have
	\begin{equation}
		\pi^* \tilde{h}
			= \pi^* (q_1)_!(q_2^*h)
			= (p_1)_!(\Pi^*q_2^* h)
			= (p_1)_!(\ev^* h)=\hat h,
	\end{equation}
	where we used the naturality of fiber integration~\cite[Def.~38]{BB}.
\end{proof}
Note that the \( \Diff_+(S) \)-invariance of the differential character $\hat h$ is not enough to conclude that \( \hat h \) descends to a differential character \( \tilde{h} \) on \( \Gr^S(M) \).
In fact, we are not aware of a direct proof that \( \hat{h} \) descends without using transgression to \( \Gr^S(M) \) as defined in~\eqref{tm}.
Even for differential forms invariance is not enough to conclude that they descend to the base (they have to be basic!).

%%%%%%
%%%%%%

\section{Integration of Lichnerowicz cocycles}

Let $M$ be a closed, connected manifold of dimension $n\geq 2$, and let 
$\mu\in\Om^n_\ZZ(M)$ be an integral volume form.
Then each oriented codimension two submanifold $N\subset M$ determines a 2-cocycle 
on the Lie algebra $\X_{\ex}(M,\mu)$ of exact divergence free vector fields by
\begin{equation}\label{L}
\ps_N(X,Y):=\int_N i_X i_Y\mu.
\end{equation} 
If $[\et]\in H^2_{\rm dR}(M)$ is Poincar\'e dual to $[N]\in H_{n-2}(M)$,
then~\eqref{L} is cohomologous to the \emph{Lichnerowicz cocycle} \cite{Lichnerowicz}
$\ps_\et(X,Y):=\int_M\et(X,Y) \, \mu$, cf.~\cite{Vizman1}.

In~\cite{DJNV}, we used transgression of differential characters over mapping spaces 
to integrate these Lie algebra cocycles  to smooth central extensions of the Lie group $\Diff_{\ex}(M,\mu)$.
In this section we apply ideas from~\cite{Ismagilov,HV04} to show that the same can be done using transgression over 
nonlinear Gra\ss{}mannians, and we indicate the relation between these two methods.

\subsection{Construction using the nonlinear Gra\ss{}mannian \texorpdfstring{$\Gr^{S}(M)$}{GrSN(M)}}
First we turn to the use of transgression over nonlinear Gra\ss{}mannians.
Here $S$ is a closed, oriented manifold of dimension $n-2$.
Let $h\in \widehat{H}^{n-1}(M,\TT)$ be a differential character on $M$ with curvature the integral volume form $\mu$,
\ie a group homomorphism $h:Z_{n-1}(M)\to\TT$ that assigns to each boundary its enclosed volume modulo $\ZZ$.
Using the transgression diagram
\begin{equation}
\begin{tikzcd}[row sep=small]
 & \T \ar[dr, "q_2"] \ar[dl, swap, "q_1"]&\\
\Gr^{S}(M)& & M,
\end{tikzcd}
\end{equation}
we saw that $h$ yields
$
\widetilde{h} := (q_1)_! (q_{2}^*h)
$ 
in 
$
\widehat{H}^1(\Gr^{S}(M),\TT),
$
a differential character with curvature 
$
\widetilde{\mu} := (q_1)_! (q_2^*\mu)
$ 
in 
$
\Omega^2_{\Z}(\Gr^S(M))
$.
Since differential characters of degree one correspond to isomorphism classes of principal 
circle bundles with connection (see Remark~\ref{one}), this yields an isomorphism class of principal circle bundles $\cP \rightarrow \Gr^{S}(M)$ equipped with a connection 1-form $\Theta_{\cP} \in \Omega^1(\cP)$ 
whose curvature is $\widetilde{\mu}$. 
In fact, the closed 2-form $\tilde\mu$ is symplectic~\cite{Marsden-Weinstein,Ismagilov,HV04}, so that $\cP \rightarrow \Gr^{S}(M)$
is a prequantum circle bundle.
As  any two differential characters  $h$ and $h'$ with curvature $\mu$ differ by
an element of  $H^{n-1}(M,\TT)$,
we obtain a distinguished class of `transgressed' prequantum bundles of \( (\Gr^S(M), \tilde{\mu}) \), forming a  
 torsor over $H^{n-1}(M,\TT)$.

Let $\Gr^{S}_{N}(M)$ be the connected component of $N\in \Gr^{S}(M)$ and let \( \cP_N \) denote the restriction of \( \cP \) to \( \Gr^{S}_{N}(M) \).
The quantomorphism group $\Aut\bigl(\cP_N,\Theta_{\cP}\bigr)$ is then a central extension 
\begin{equation}\label{eq:preqCE}
 \TT \to \Aut\bigl(\cP_N,\Theta_{\cP}\bigr) \to
\Diff\bigl(\Gr_{N}^{S}(M),\tilde h\bigr)
\end{equation}
of the group $\Diff\bigl(\Gr_{N}^{S}(M),\tilde h\bigr)$ of holonomy-preserving diffeomorphisms by the circle group $\TT$, cf.~\cite{Kostant,Souriau}, and~\cite{NV} for the infinite dimensional case.

The group $\Diff_{\ex}(M,\mu)$ in Example~\ref{exem} (with identity component the group of exact volume-preserving diffeomorphisms) might be non-connected.
Being a subgroup of $\Diff(M,\mu)_0$, by continuity, the natural action \( \sigma \) of \( \Diff_{\ex}(M,\mu) \subseteq \Diff(M)_0 \) leaves the connected component \( \Gr^S_N(M) \) invariant.
Moreover, Proposition~\ref{prop:transgressionDiffeoEquiv} implies that \( \Diff_{\ex}(M,\mu) \subseteq \Diff(M, h) \) preserves \( \tilde h \).
Thus, a central group extension of $\Diff_{\ex}(M,\mu)$ can be obtained by pull-back of the prequantization central extension~\eqref{eq:preqCE} by the action $\si$,
\begin{equation}  \label{eq:diag}
\xymatrix{
1\ar[r]&\TT\ar[r] &\Aut\bigl(\cP_N,\Theta_{\cP}\bigr)\ar[r]^{\ } &
\Diff\bigl(\Gr_{N}^{S}(M),\tilde h\bigr)\ar[r]&1\\
1\ar[r]&\TT\ar[u]\ar[r]& \widehat{\Diff}_{\ex}(M,\mu)\ar[u] \ar[r] &
\Diff_{\ex}(M,\mu)\ar[u]^\si\ar[r]&1.\\
}
\end{equation}
The group $\Aut\bigl(\cP_N,\Theta_{\cP}\bigr)$ need not be a locally convex Lie group. But
it follows from the generalization of~\cite[Thm.~3.4]{NV} to non-connected Lie groups given in~\cite[Thm.~A.1]{DJNV} that
the pull-back $\widehat\Diff_{\ex}(M,\mu)$ 
is  a Lie group, with the manifold structure coming from the pull-back of $\cP_N \rightarrow \Gr^{S}_{N}(M)$
along the orbit map $\Diff_{\ex}(M,\mu)\rightarrow \Gr^{S}_{N}(M)$,
cf.~\cite[Rk.~4]{HV04}.

Although~\eqref{eq:preqCE} is not a central extension of locally convex Lie groups, its 
Lie algebra extension in the sense of Definition~\ref{Def:FakeLieAlgebra} is the one with the Kostant--Souriau cocycle
\[
 \psi_{KS}(\widetilde{Y}_1,\widetilde{Y}_2) = \widetilde{\mu}_{N}(\widetilde{Y}_1,\widetilde{Y}_2), 
 \quad \widetilde{Y}_1,\widetilde{Y}_2 \in T_{N}\Gr_{N}^{S}(M) = \Gamma(TN^{\perp}).
\]
The central Lie algebra extension corresponding to the Lie group extension \( \widehat\Diff_{\ex}(M,\mu) \to \Diff_{\ex}(M,\mu) \) in \eqref{eq:diag} is the pull-back along the infinitesimal action $\sigma_* \colon \X_{\ex}(M,\mu) \rightarrow \X(\Gr_{N}^{S}(M),\widetilde{\mu})$, \( X \mapsto \tilde{X} \) of the Kostant--Souriau cocycle,
\[
 \psi_{KS}(\sigma_*X,\sigma_*Y) = \widetilde{\mu}_{N}(\tilde{X},\tilde{Y})
 \stackrel{\eqref{fun}}{=}\widetilde{i_Y i_X\mu}(N)\stackrel{\eqref{tide}}{=}\int_N i_Y i_X\mu 
\stackrel{\eqref{L}}{=}\ps_N(X,Y).
\]
It follows that $\TT \rightarrow \widehat\Diff_{\ex}(M,\mu)\rightarrow \Diff_{\ex}(M,\mu)$ 
is a central extension of Fr\'echet--Lie groups integrating the Lichnerowicz cocycle $\ps_N$ defined in~\eqref{L}. 
We have thus proven the following result:

\extensionsIntegratingLichnerowicz*

\begin{rema}
	In~\cite{Ismagilov,HV04} another approach to integrate the Lichnerowicz cocycle $\ps_N$ to smooth central extensions of the group of exact volume-preserving diffeomorphisms, \ie the identity component of $\Diff_{\ex}(M,\mu)$, is presented.
	We do not know whether these extensions coincide with the extensions given in Theorem~\ref{extensionsIntegratingLichnerowicz}.
	The construction in~\cite{HV04} also uses a prequantum bundle over $\Gr^{S}(M)$, which is constructed by hand in a rather complex process through its \v Cech 1-cocycle.
	The main novelty of the present work is the use of differential characters to obtain the prequantum bundle over $\Gr^{S}(M)$ that is needed for the construction of the smooth central extension \( \widehat\Diff_{\ex}(M,\mu) \).
\end{rema}

\subsection{Construction using the embedding space \texorpdfstring{$\Emb(S,M)$}{Emb(S, M)}}
\label{sec:mapping}

In~\cite{DJNV}, we constructed a central extension of $\Diff_{\ex}(M,\mu)$ using transgression to the mapping space $C^{\infty}(S,M)$, where $S$ is a closed, oriented manifold of dimension $n-2$.
We now briefly recall the construction in~\cite{DJNV} and adapt it to the  case of embeddings. 
Using the transgression diagram
\begin{equation}
\xymatrix@C-2pc@R-1pc{
&{\Emb(S,M) \x S}\ar[dr]_{\ev}\ar[dl]^{p_1}&\\
\Emb(S,M)& & 
M,
}
\end{equation}
a differential character $h\in \widehat{H}^{k-1}(M,\TT)$ with curvature $\mu$ transgresses to 
the differential character
$\hat{h} := (p_1)_! (\ev^*h)$ 
in 
$\widehat{H}^{1}(\Emb(S,M),\TT)$,
with curvature 
$\hat{\mu} := (p_1)_! (\ev^*\mu)$ 
in
$\Omega^2_{\Z}(\Emb(S,M))$, see \cref{sec:transgression_of_differential_characters}.
This yields a principal $\TT$-bundle $\cQ \rightarrow \Emb(S,M)$
with connection $\Theta_{\cQ}$ and curvature $\hat{\mu}$.

For a smooth embedding $f \colon S \rightarrow M$, let 
$\Emb_{f}(S,M)$ be the connected component of $f$ in $\Emb(S,M)$ and let \( \cQ_f \) denote the restriction of \( \cQ \) to \( \Emb_f(S,M) \). 
The group of connection-preserving automorphisms of $\cQ_f\rightarrow{\Emb_{f}(S,M)}$ is a central extension of the group of $\hat{h}$-preserving diffeomorphisms of $\Emb_f(S,M)$.
If we pull this back along the action 
$\sigma \colon \Diff_{\ex}(M,\mu) \to \Diff(\Emb_f(S,M),\hat{h})$
of $\Diff_{\ex}(M,\mu)$ on $\Emb_{f}(S,M)$, we obtain a central
extension $\widehat{\Diff}{}'_{\ex}(M,\mu)$ of $\Diff_{\ex}(M,\mu)$ by $\TT$,
\begin{equation*}\label{big2}
\xymatrix{
1\ar[r]&\TT\ar[r] &\Aut(\cQ_f,\Theta_{\cQ}) \ar[r] & \Diff(\Emb(S,M),\hat{h}) \ar[r]&1\\
1\ar[r]&\TT\ar[r]\ar[u]& \widehat{\Diff}{}^{'}_{\ex}(M,\mu)\ar[u] \ar[r] &
\Diff_{\ex}(M,\mu)\ar[u]_\si\ar[r]&1.
}
\end{equation*}
The group $\widehat{\Diff}{}'_{\ex}(M,\mu)$ is a Fr\'echet--Lie group, and the Lie algebra
\( 2 \)-cocycle corresponding to this central extension is the Lichnerowicz cocycle~\eqref{L} with \( N = f(S) \) by~\cite[Thm.~5.4]{DJNV}.

\subsection{Comparison}\label{sec:comparison}
Since the corresponding Lie algebra cocycles coincide, the two central extensions $\widehat{\Diff}_{\ex}(M,\mu)$ and $\widehat{\Diff}{}'_{\ex}(M,\mu)$ constructed using transgression, respectively, on the nonlinear Gra\ss{}mannian $\Gr_{N}^{S}(M)$ and on the embedding space $\Emb_{f}(S,M)$, are isomorphic on the infinitesimal level if \( f(S) = N \).
To show that they are isomorphic also as central extensions of Lie groups, we will use the following general result.
\begin{lemm}\label{general}
	Assume that a Lie group \( G \) acts on the connected manifolds \( M \) and \( N \).
	Let \( P \to M \) be a principal \( \TT \)-bundle with connection \( \theta \) whose holonomy is \( G \)-invariant.
	For a  \( G \)-equivariant map \( \psi: N \to M \), let \( \hat{G} \) be the central extension of \( G \) obtained from the bundle \( (P, \theta) \) and let \( \hat{G}_\psi \) the one obtained from the pull-back bundle \( (\psi^* P, \psi^* \theta) \).
	For every \( \phi \in \Aut(P, \theta) \) covering the action of some \( g \in G \), the map
	\begin{equation}
		\label{eq:comparsion:inducedAutomorphism}
		\overline{\phi}: \psi^* P \to \psi^* P, \qquad (n, p) \mapsto \bigl(g \cdot n, \phi(p)\bigr).
	\end{equation}
	is a bundle automorphism of \( \psi^* P \) which preserves the pull-back connection \( \psi^* \theta \) and covers the action of \( g \in G \) on \( N \).
	The resulting Lie group homomorphism \( \hat{G} \ni \phi \mapsto \overline{\phi} \in \hat{G}_\psi \) yields a smooth isomorphism of central extensions.
\end{lemm}
\begin{proof}
	Let \( \phi \in \Aut(P, \theta) \) covering the action of \( g \in G \).
	By the \( G \)-equivariance of \( \psi \), the prescription~\eqref{eq:comparsion:inducedAutomorphism} indeed defines a smooth bundle map \( \overline{\phi}: \psi^* P \to \psi^* P \) that covers the action of \( g \) on \( N \) by construction.
	It is straightforward to see that \( \overline{\phi} \) is a bundle automorphism preserving the connection \( \psi^* \theta \), as \( \phi \in \Aut(P, \theta) \).
	Clearly, \( \hat{G} \ni \phi \mapsto \overline{\phi} \in \hat{G}_\psi \) is a group homomorphism fitting into the commutative diagram
	\begin{equation}
		\xymatrix{
			1\ar[r]	&\TT \ar[r] 			&\hat{G}_\psi \ar[r] 	& G \ar[r]			&1\\
			1\ar[r]	&\TT \ar[r]\ar[u]_\id	&\hat{G} \ar[u]\ar[r] 	& G \ar[u]_\id\ar[r]&1
		}
	\end{equation}
	and it thus yields an isomorphism of central extensions~\cite[Def.~V.1.1]{Neeb2006}.
\end{proof}
\begin{coro}
	\label{prop:comparison:withDiffCharacter}
	Assume that a Lie group \( G \) acts on the connected manifolds \( M \) and \( N \).
	Let \( h_M \in \widehat{H}^1(M,\TT) \) and \( h_N \in \widehat{H}^1(N,\TT) \) be \( G \)-invariant differential characters.
	If there exists a \( G \)-equivariant map \( \psi: N \to M \) such that \( \psi^* h_M = h_N \), then the central extensions of \( G \) obtained from \( h_M \) and from \( h_N \) are isomorphic.
\end{coro}

Let us return to the main objective of comparing the central extensions $\widehat{\Diff}_{\ex}(M,\mu)$ and $\widehat{\Diff}{}'_{\ex}(M,\mu)$ constructed using transgression on the nonlinear Gra\ss{}mannian $\Gr_{N}^{S}(M)$ and on the embedding space $\Emb_{f}(S,M)$.
Here and in the following, \( f \) and \( N \) are related via \( f(S) = N \).
For \( h\in \widehat{H}^{n-1}(M,\TT) \), the transgressed differential characters $\widehat{h} \in \widehat{H}^1(\Emb(S,M),\TT)$ and $\widetilde{h} \in \widehat{H}^1(\Gr^{S}(M),\TT)$ are related by 
$\widehat{h} = \pi^*\widetilde{h}$ according to Proposition~\ref{prop:piRelationOfTransgressions}.
Since \( \pi \) is \( \Diff_{\ex}(M, \mu) \)-equivariant, Corollary~\ref{prop:comparison:withDiffCharacter} yields the following comparison result.
\begin{prop}
	Let $M$ be a compact manifold of dimension \( n \) endowed with a differential character $h\in \widehat{H}^{n-1}(M,\TT)$ with curvature $\mu$.
	Let \( S \) be a closed, oriented manifold of dimension \( n-2 \).
	For every embedding \( f: S \to M \), the central extensions $\widehat{\Diff}_{\ex}(M,\mu)$ and $\widehat{\Diff}{}'_{\ex}(M,\mu)$ constructed using transgression to the nonlinear Gra\ss{}mannian $\Gr_{f(S)}^{S}(M)$ and to the embedding space $\Emb_{f}(S,M)$ are smoothly isomorphic.
\end{prop}

In particular, it follows that the extensions obtained from $\cP \rightarrow \Gr^{S}(M)$ over different connected components $\Gr_{N}^{S}(M)$ and $\Gr_{N'}^{S}(M)$ are isomorphic as soon as $N$ and $N'$ are the images of homotopic embeddings.
Furthermore, if a submanifold $N \subseteq M$ is the image of an embedding $f \colon S \rightarrow M$
that is homotopic to a `thin' map $g \colon S \rightarrow M$ (meaning that $\int_{S}g^*\alpha =0$ for all 
$\alpha \in \Omega^{n-2}(M)$), then the corresponding central extension
$\widehat{\Diff}_{\ex}(M,\mu)$ is trivial at the Lie algebra level.

\begin{rema}[Comparison of the constructions]
As \( \pi: \Emb(S, M) \to \Gr^{S}(M) \) is surjective, every extension of \( \Diff_{\ex}(M,\mu) \) that can be obtained using the nonlinear Gra\ss{}mannian can also be obtained starting from the embedding space, showing that the two approaches are essentially equivalent.
An advantage of the construction using Gra\ss{}mannians is that the trangressed 2-form $\widetilde{\mu}$ 
on $\Gr^{S}(M)$ is not only 
closed, but also nondegenerate; in contrast to the principal circle bundle over $\Emb(S,M)$, the circle bundle $\cP_{N}$ (with connection $\Theta_{\cP}$) over 
$\Gr_{N}^{S}(M)$ (with symplectic form $\tilde{\mu}$) is thus a prequantum bundle.
Under certain conditions, the Grassmannian $\Gr_{N}^{S}(M)^{+}$ of oriented embedded submanifolds admits a 
formally integrable almost complex structure \cite{Brylinski, Lempert1993, LeBrun1993, Verbitsky2012, FL21}.
In future work, we hope to use this almost complex structure in order to define a suitable polarization on the associated prequantum line bundle $\mathbb{L} = \cP_{N}\times_{\TT}\mathbb{C}$,
making the step from prequantization to quantization.
\end{rema}

\begin{rema}[Construction based on \( C^{\infty}(S,M) \)]
	In~\cite{DJNV}, we constructed a central extension of $\Diff_{\ex}(M,\mu)$ using transgression to the mapping space $C^{\infty}(S,M)$ instead of the embedding space \( \Emb(S, M) \) as described in \cref{sec:mapping}.
	Clearly, the inclusion \( \iota: \Emb(S, M) \to C^{\infty}(S,M) \) is \( \Diff_{\ex}(M, \mu) \)-equivariant and relates the transgressed characters $\widehat{h} \in \widehat{H}^1(\Emb(S,M),\TT)$ and $\overline{h} \in \widehat{H}^1(C^{\infty}(S,M),\TT)$ by $\widehat{h} = \iota^*\overline{h}$.
	Thus, Corollary~\ref{prop:comparison:withDiffCharacter} implies that the central extensions of \( \Diff_{\ex}(M,\mu) \) obtained from \( (\Emb_f(S, M), \widehat{h}) \) and from \( (C^{\infty}_f(S,M), \overline{h}) \) are smoothly isomorphic for every embedding \( f \in \Emb(S, M) \).
	However, there could be connected components of \( C^{\infty}(S,M) \) that do not contain an embedding\footnote{For example,~\cite[Corollary~1.3]{Kronheimer1997} yields immersions of a closed surface into a closed \( 4 \)-manifold that are not homotopic to an embedding.}, and such connected components might give rise to central extensions of $\Diff_{\ex}(M,\mu)$ which are not accessible via embedding spaces and nonlinear Gra\ss{}mannians.
\end{rema}

\subsection{The special case \texorpdfstring{$H^1(M,\Z)=\{0\}$}{H1(M,Z)}}
In this section we treat the special case when 
$H^1(M,\Z)=\{0\}$. 
The group $\Diff_{\ex}(M,\mu)_0$ 
of exact volume-preserving diffeomorphisms then coincides with 
the  smooth arc-component of the identity 
of the group $\Diff(M,\mu)$ of volume-preserving diffeomorphisms.
Indeed, $H^{n-1}_{\rm dR}(M) = \{0\}$, since  by Poincar\'e duality
$H_{n-1}(M,\Z) \cong H^1(M,\Z)=\{0\}$.
Since $H^{n-1}(M,\TT)=\{0\}$, it now follows from the exact sequence~\eqref{eq:1} that 
$\curv:\hat H^{n-1}(M,\TT)\to\Om^n_{\Z}(M)$ is an isomorphism.
In particular, to any volume form $\mu$ with integral periods
one can assign in a natural way a uniquely determined 
differential character $h_\mu\in\hat H^{n-1}(M,\TT)$ with curvature $\mu$.
It is defined as $h_{\mu}(c) := \exp_{\TT}(\int_{D}\mu)$ for any 
$c\in Z_{n-1}(M)$ and
$D\in C_{n}(M)$ with $\partial D = c$. Such smooth singular chains  $D$ exist because the smooth singular homology group $H_{n-1}(M) = 0$, 
and the result is independent of the choice of $D$, because $\mu$ is integral.

If $H^1(M,\Z) = \{0\}$, the diffeomorphism groups that preserve 
the volume form $\mu$ and the holonomy $h_\mu$ coincide:
$\Diff(M,\mu)=\Diff(M,h_\mu)$ by the right commutative diagram in~\eqref{lambda}.
We thus have
$$\Diff_{\ex}(M,\mu) \stackrel{\eqref{opt}}{=} \Diff(M,\mu)_{0} = \Diff_{\ex}(M,\mu)_{0} = \Diff(M,h_{\mu})_{0}.$$
We can therefore canonically associate to every homotopy class $[f]$ of an embedding $f \colon S \rightarrow M$ 
the isomorphism class of a central Lie group extension
\[
 \TT \rightarrow \widehat{\Diff}(M,\mu)_{0} \rightarrow \Diff(M,\mu)_0 
\]
with cocycle 
\(\psi(X,Y) = \int_{S} f^*(i_{X}i_{Y}\mu)\).

\begin{exam}
When $M$ is a surface with $H_1(M)=\{0\}$ and $S=\{*\}$ one point, then the nonlinear Gra\ss{}mannian is $\Gr^S(M)=M$ and 
the transgression map is the identity.
Thus the action $\si$ is the identity,
and the two rows in~\eqref{eq:diag} coincide.

For a 2-sphere $M=S^2$ the group of volume-preserving diffeomorphisms is connected: 
it has the rotation group ${\rm SO}(3)$ as a deformation retract by~\cite{Smale}.
Thus $\Diff_{\ex}(S^2,\mu)_0=\Diff(S^2,\mu)$ for any volume form $\mu$ on the 2-sphere.
Since $H^1(S^2,\ZZ)=\{0\}$, each $\mu\in\Om_\ZZ^2(S^2)$
determines uniquely a differential character $h_\mu\in\hat H^1(S^2,\TT)$.
It is the holonomy of a suitable `Hopf fibration' $(S^3,\th)\to(S^2,\mu)$,
a principal circle bundle with connection $\th$ and curvature~$\mu$.
\end{exam}
The next example shows that every knot on the \( 3 \)-sphere yields a central extension of the identity component of the volume-preserving diffeomorphism group.
\begin{exam}
On the 3-sphere $S^3 \simeq \SU(2)$ we consider the bi-invariant volume form 
\[
\mu=\kappa(\th\wedge[\th\wedge\th]),
\]
where $\th$ denotes the Maurer-Cartan $1$-form on $\SU(2)$.
We scale the Killing form \( \kappa \) such that the cohomology class $[\mu]$ of the volume form
generates the image of $H^3(\SU(2),\Z)\simeq\Z$ in $H^3(\SU(2),\R)$. 
Let $h_\mu$ be the unique differential character with curvature $\mu$.
Each embedded knot in \( S^3 \)
yields a central Lie group extension of $\Diff(S^3,\mu)_0$.
Since $H^2(S^3)=0$, the corresponding  extension of the Lie algebra of divergence free vector fields is trivial~\cite{Ro95}.
\end{exam}
In the next example, the second cohomology of the ambient space is non-trivial.
\begin{exam}
Let $M=S^2\x S^2$ be endowed with the volume form $\mu=p_1^*\nu\wedge p_2^*\nu$,
with $\nu$ an area form on $S^2$ of total area $1$.
Again we have a unique differential character $h_\mu$ on $S^2\x S^2$ with the volume form $\mu$ with integral periods as curvature.
The isomorphism classes of central extensions of the Lie algebra of divergence free vector fields on $S^2\x S^2$ 
are classified by $H^2(S^2\x S^2)\simeq \R^2$ \cite{Ro95}.

For each embedded $S^2$ in $S^2\x S^2$ one gets in a canonical way
a central extension of the identity component of the group of volume-preserving diffeomorphisms of $S^2\x S^2$.
For instance the diagonal embedding yields a Lie group extension that integrates a non-trivial Lie algebra extension:
the Lichnerowicz 2-cocycle
$\ps(X,Y)=\int_{\Delta_{S^2}}i_Xi_Y\mu$ has non-trivial cohomology class.
\end{exam}

%%%%%%%%%%%
%%%%%%%%%%%%

%%%%%%%%%%%%%%%%%
%%%%%%%%%%%%%%%%%

\vspace*{7ex}

\noindent Tobias Diez. Institute of Applied Mathematics, Delft University of Technology, 2628 XE Delft, The Netherlands. \texttt{t.diez@tudelft.nl}
\\[1ex]
\noindent Bas Janssens. Institute of Applied Mathematics, Delft University of Technology, 2628 XE Delft, The Netherlands. \texttt{b.janssens@tudelft.nl}
\\[1ex]
\noindent Karl-Hermann Neeb. Department of Mathematics, FAU Erlangen-N\"urnberg, 91058 Erlangen, Germany. \texttt{neeb@math.fau.de}
\\[1ex]
\noindent Cornelia Vizman. Department of Mathematics, West University of Timi\c soara. RO--300223 Timi\c soara. Romania. \texttt{cornelia.vizman@e-uvt.ro}

\end{document}